\documentclass[a4paper,10pt,reqno]{amsart}
\usepackage{vmargin}
\usepackage[pdftex]{graphicx}
\usepackage{mathtools}
\usepackage{amsfonts}
\usepackage{amssymb}
\usepackage{amsrefs} 
\usepackage{mathrsfs}
\usepackage{stmaryrd}
\usepackage{amsmath}
\usepackage{amsthm}
\usepackage[shortlabels]{enumitem}
\usepackage{nomencl}
\usepackage[bookmarks=true]{hyperref}   
\usepackage{esint}
\usepackage{dsfont}
\usepackage{overpic}
\usepackage{upgreek}
\usepackage[dvipsnames]{xcolor}
\usepackage{slashed}
\usepackage{scalerel}
\usepackage{comment}
\usepackage[utf8]{inputenc}
\usepackage{faktor}
\usepackage{nicefrac}
\usepackage[textwidth=30mm, textsize=tiny]{todonotes}

\hypersetup{
    colorlinks,%
}

\author{Christian Bär}
\author{Lashi Bandara}
\title{A hitchhiker's guide to first-order elliptic boundary value problems}
\date{\today}
\address{Christian Bär, 
Institut für Mathematik,
Universität Potsdam, 
D-14476 Potsdam, Germany
}
\urladdr{\href{https://www.math.uni-potsdam.de/baer}{https://www.math.uni-potsdam.de/baer}}
\email{\href{mailto:christian.baer@uni-potsdam.de}{christian.baer@uni-potsdam.de}}

\address{Lashi Bandara,
Deakin University, 
Melbourne Burwood Campus, 
221 Burwood Highway, Burwood, Victoria, Australia, 3125
}
\urladdr{\href{https://lashi.org}{https://lashi.org}}
\email{\href{mailto:lashi.bandara@deakin.edu.au}{lashi.bandara@deakin.edu.au}}


\def\colour{\colour}

\newcommand{\sym}{\upsigma}
\newcommand{\rest}[1]{{{\lvert_{}}_{}}_{#1}}
\newcommand{\R}{\mathbb{R}}
\newcommand{\C}{\mathbb{C}}
\newcommand{\N}{\mathbb{N}}
\newcommand{\In}{\mathbb{Z}}

\newcommand{\dbrac}[1]{\left\{#1\right\}}
\newcommand{\ad}{\ast}

\newcommand{\id}{\mathrm{id}}

\newcommand{\modulus}[1]{|#1|}

\newcommand{\norm}[1]{\| #1 \|}			

\newcommand{\set}[1]{\dbrac{#1}}
 
\newcommand{\Lp}[2][{}]{{\rm L}^{#2}_{\rm #1}}		
\newcommand{\Ck}[2][{}]{{\rm C}^{#2}_{\rm #1}}		
\newcommand{\Hard}[2][{}]{{\rm H}^{#2}_{\rm #1}}		
\newcommand{\SobH}[2][{}]{\Hard[#1]{#2}}

\DeclareMathOperator{\Ind}{index}

\newcommand{\embed}{\hookrightarrow}		

\newcommand{\dom}{\mathrm{dom}}

\newcommand{\intersect}{\cap} 
\newcommand{\close}[1]{\overline{#1}}

\DeclareMathOperator{\spt}{spt}		

\newcommand{\inprod}[1]{\left\langle #1 \right\rangle}	
\newcommand{\ext}{\mathcal{E}}

\newcommand{\APS}{\mathrm{APS}}

\renewcommand{\epsilon}{\varepsilon}

\newcommand{\Hinfty}{$\mathrm{H}^\infty$}
\newcommand{\dM}{\partial M}

\renewcommand{\emptyset}{\varnothing}
\newcommand{\Bmatch}{B_{\rm M}}
\newcommand{\Poiss}{\mathcal{K}} 

\newcommand{\Dt}{\mathscr{D}}

\newtagform{simple}{}{}{}

 \newtheorem{theorem}{Theorem}
 \newtheorem{lemma}{Lemma}
 
 \newtheorem{corollary}{Corollary}

 \theoremstyle{definition}
 \newtheorem{definition}{Definition}
 \newtheorem{example}{Example}
 \newtheorem{remark}{Remark}

\begin{document}

\maketitle
\begin{abstract}
To empower the mathematical hitchhiker wishing to use operator methods in geometry and topology, we present this user's guide to first-order elliptic boundary value problems.
Existence, regularity, and Fredholmness are discussed for general first-order elliptic operators on  manifolds with compact boundary. 
The focus is on a very general class of elliptic boundary conditions, which contain those that are pseudo-local as a special case, yielding the relative index theorem.
A new characterisation of a subclass of elliptic boundary conditions is also given.

\end{abstract}
\tableofcontents

\parindent0cm
\setlength{\parskip}{\baselineskip}

\newcounter{alphsection}
\renewcommand{\thealphsection}{\Alph{alphsection}}
\newcommand{\alphsection}[1]{%
   \refstepcounter{alphsection}%
   \section*{\thealphsection\quad #1}%
   \markright{#1}%
}

\section*{Introduction}

The index theorem of Atiyah-Patodi-Singer \cite{APS-Ann, APS1, APS2, APS3} for Dirac operators on smooth compact manifolds with boundary is heralded today as a major mathematical achievement of the twentieth century.
This result, beyond its immediate value, highlighted  non\-local boundary conditions as the quintessential type in the study of first-order boundary value problems. 
Their study has been a primary focus in the decades since, with a particular focus given to pseudo-local boundary conditions, arising as the range of a pseudo-differential projector of order zero.
 
Although it is beyond the scope of this paper to provide an exhaustive list of contributions, \cite{BB12, BBGuide, BBC, BLZ, B, BL2001, G96, Melrose, RS, S01, S04} by Bär, Ballmann, Booß-Bavnbek, Boutet de Monvel, Brüning, Carron, Chen, Grubb, Lesch, Melrose, Rempel, Schulze, and Zhu are a list of references which has direct relevance to what we present here.
A typical assumption in all of these papers is that the adapted boundary operator, which can be thought of as the trace of the operator to the boundary, can be chosen self-adjoint. 
In particular, \cite{BB12} provides a description of \emph{all} boundary conditions. 
That is, the boundary trace map is extended to the whole of the maximal domain. 
Furthermore, regularity, Fredholmness, and index theory are discussed in a broadly applicable context. 

In \cite{BBan} by Bär-Bandara, the technical requirement in \cite{BB12} (and in other earlier works) requiring a self-adjoint adapted operator, is dispensed.
In fact, results of \cite{BBan} are very general - they can be applied to \emph{general} first-order elliptic operators on manifolds with compact boundary.
The methods employed in \cite{BBan} deviate from earlier works using Fourier circle methods. 
Instead, modern \Hinfty-functional calculus methods intertwined with real-variable harmonic analysis techniques are used to tame non-self-adjoint adapted boundary operators. 
These techniques are considerably technical in nature.

Let us now arrive at the present paper, which we introduce with the following analogy. 
The roadside hitchhiker, in order to travel to their desired destination, need not know about mechanical aspects of motor vehicles nor do they even need to know how to drive.
Much in the same way, the mathematical hitchhiker should be able to utilise results in \cite{BBan} to achieve their desired mathematical destiny, without the burden of labouring through technicalities. 
It is in this spirit that this ``user's guide'' to first-order boundary value problems has been conceived.

The structure of this paper is as follows. 
In Section~\ref{Sec.Setup}, the minimal and maximal extensions, along with a \emph{standard setup} \ref{Hyp.StdFirst}-\ref{Hyp.StdLast} under which results are obtained, are given.
Examples of significance are provided which may assist the hitchhiker in their own calculations.
The short  Section~\ref{Sec.Complete} is dedicated to discussing and presenting a very natural method to help the hitchhiker to verify the so-called \emph{completeness} assumption~\ref{Hyp.Complete}.

Section~\ref{Sec.EBC} contains the central objects of this paper - elliptic boundary conditions.
Specialising results of \cite{BBan} to a frequently encountered smooth setting, the notion of an $\infty$-elliptic boundary condition is given in Definition~\ref{Def.ER}.
This is an important notion which characterises such boundary conditions in a \emph{graphical} form, an incredibly flexible and powerful tool for analysis of problems in topology and geometry.
In order to utilise this notion, results pertaining to boundary regularity, the adjoint boundary condition, and the relationship to classical pseudo-local boundary conditions are presented. 

Section~\ref{Sec.Fred} introduces the notion of coercivity which guarantee $\infty$-elliptic boundary conditions to yield a Fredholm operator.
Related to these ideas, Section~\ref{Sec.Match} introduces the \emph{matching} boundary condition, an example of an $\infty$-elliptic boundary condition which is not pseudo-local. 
This is a crucial boundary condition used to obtain the relative index theorem in this generality, a result also included in this section.

In Section~\ref{Sec.HighReg} the notion of a $k$-elliptic boundary condition (in the sense of \cite{BBan}*{Definition 2.11}) is characterised by the regularity of solutions subjected to that boundary condition. 
This is a new and useful characterisation which was recently obtained and not included in \cite{BBan}.

Lastly, Appendix~\ref{appendix.RS} contains a calculation of the ellipticity of the Rarita-Schwinger operator.
Much of the development in \cite{BBan} was motivated by the desire to analyse this operator, which arises naturally in geometry. 
This is not of Dirac-type and, in fact, adapted boundary operators induced from the Rarita-Schwinger operator are generally non-self-adjoint. 
This calculation is included to provide scaffolding to potential calculations which the hitchhiker may need to perform in their own context.

\subsection*{Acknowledgements} 
This work was financially supported by the Schwerpunktpro\-gramm~2026 ``Geometry at Infinity'' funded by \emph{Deutsche Forschungsgemeinschaft}.
L.B. would like to thank Magnus Goffeng for useful conversations about regularity.

\section{Setup and preliminaries}
\label{Sec.Setup} 
Throughout, $M$ will be a smooth  manifold with smooth boundary.
We write $\Ck{k}(M;E)$ to denote the space of $k$-times continuously differentiable sections of $E$, $\Ck[c]{k}(M;E)$ the subspace  of compactly supported sections (possibly nonzero on the boundary), and  $\Ck[cc]{\infty}(M;E)$ the subspace of $\Ck[c]{\infty}(M;E)$ whose sections are supported on the interior of $M$.

We fix a smooth measure $\mu$ on $M$. 
By this, we mean a smooth positive section of the density bundle of $M$.
Given a Hermitian vector bundle $(E, h^E) \to M$, we naturally obtain the Hilbert space $\Lp{2}(M;E)$ of square integrable sections.

When $M$ is compact and without boundary, the Sobolev spaces, with respect to $\Lp{2}$, are denoted by $\SobH{\alpha}(M;E)$, where $\alpha\in\R$.
These are Hilbert spaces.
For $\alpha<\beta$, there is a continuous embedding $\SobH{\beta}(M;E) \embed \SobH{\alpha}(M;E)$.
In particular, $\SobH{0}(M;E)=\Lp{2}(M;E)$ and, for $\alpha\in\N$, the elements of $\SobH{\alpha}(M;E)$ are the sections whose distributional derivatives up to order $\alpha$ lie in $\Lp{2}(M;E)$.
For each $\alpha\in\R$, extending the $\Lp{2}$-scalar product in one argument and restricting in the other, we obtain a perfect pairing $\inprod{\cdot,\cdot}_{\SobH{\alpha}\times \SobH{-\alpha}}\colon \SobH{\alpha}(M;E)\times \SobH{-\alpha}(M;E)\to\C$.

We fix a first-order linear differential operator $D\colon \Ck{\infty}(M;E) \to \Ck{\infty}(M;F)$, where $(F,h^F)\to M$ is another Hermitian bundle.
There is a unique formal adjoint $D^\dagger\colon \Ck{\infty}(M;F) \to \Ck{\infty}(M;E)$.
The  maximal and minimal extensions  of $D$  are defined by
$$ 
D_{\max} := \big(D^\dagger|_{\Ck[cc]{\infty}}\big)^\ad\quad\text{ and }\quad D_{\min} := \close{\big(D|_{\Ck[cc]{\infty}}\big)},
$$ 
where $^\ad$ denotes the $\Lp{2}$-adjoint and $\close{\phantom{x}}$ the closure in $\Lp{2}(M;E)$.
The domains $\dom(D_{\max})$ and $\dom(D_{\min})$ are Banach spaces with respect to the graph norm $u\mapsto\norm{u}_D=\norm{u}_{\Lp{2}} + \norm{Du}_{\Lp{2}}$.
Similarly, we define $D^\dagger_{\max}$ an $D^\dagger_{\min}$ by interchanging the roles of $D$ and $D^\dagger$.
The principal symbol of $D$ is denoted by $\sym_D(\xi)$, which is characterised by $D(fu)=\sym_D(df)u+fDu$.

The standard setup in which we work is the following:
\begin{enumerate}[label=(S\arabic*), labelwidth=!, labelindent=2pt, leftmargin=21pt] 
\label{StdSetup}
\item \label{Hyp.StdFirst} 
	$M$ is a smooth manifold with compact smooth boundary $\dM$;
\item $\mu$ is a smooth measure on $M$;
\item \label{Hyp.Inward}
$T\in \Ck{\infty}(\dM;TM)$ is an interior pointing vector field along $\dM$; 
\item \label{Hyp.Bundles}
$(E,h^E),\ (F,h^F) \to M$ are Hermitian vector bundles over $M$;
\item \label{Hyp.D}
$D$ is a first-order elliptic differential operator mapping sections of $E$ to those of $F$; 
\item \label{Hyp.StdLast}\label{Hyp.Complete}
$D$ and $D^\dagger$ are complete, i.e., compactly supported sections in $\dom(D_{\max})$ are dense in $\dom(D_{\max})$ with respect to the graph norm $\norm{\cdot}_D$ and similarly for $D^\dagger$.
\end{enumerate}

Note that $\dM$ is assumed to be compact but we do not assume that $M$ is compact.
So the theory applies if $M$ is the complement of a relatively compact smooth domain in $\R^n$, for example.

The vector field $T$ induces a covector field $\tau\in \Ck{\infty}(\dM;T^*M)$ characterised by the conditions $\tau(T)=1$ and $\tau|_{T\dM}=0$.

The measure $\mu$ on $M$ together with $T$ induce a smooth measure $\nu$ on $\dM$ given by $\nu(\cdots) = \mu(T,\cdots)$.

\begin{example}\label{Ex:Dirac}
Let $M$ be a Riemannian manifold.
The Riemannian metric induces a measure $\mu$ on $M$ and an interior pointing unit normal field $T$ along $\dM$.

Let $(E,h^E),\ (F,h^F) \to M$ be Hermitian vector bundles over $M$ of the same rank.
A first-order differential operator $D$ mapping sections of $E$ to those of $F$ is called a \emph{Dirac-type operator} if its principal symbol satisfies the Clifford relations $\sym_D(\xi)^*\sym_D(\eta)+\sym_D(\eta)^*\sym_D(\xi) = 2g(\xi,\eta)\id_E$ for all $\xi,\eta\in T^*_xM$ and $x\in M$.
In particular, $\sym_D(\xi)^*\sym_D(\xi)=|\xi|^2\id_E$ so that $\sym_D(\xi)$ is injective for all $\xi\neq0$.
Since the ranks of $E$ and $F$ are the same, $\sym_D(\xi)$ is invertible.
Hence Dirac-type operators are elliptic.
Theorem~\ref{Thm:cherwolf} will show that $D$ and $D^\dagger$ are complete if the Riemannian metric of $M$ is complete.
\end{example}

\begin{example}\label{Ex:RS}
Let $M$ be a Riemannian manifold of dimension $n\ge3$.
Assume the setup of Example~\ref{Ex:Dirac} and let $D$ be a Dirac-type operator between $E$ and $F$.
We define a Dirac-type operator $\Dt$ between $T^*M\otimes E$ and $T^*M\otimes F$ by 
$$
\Dt(\eta\otimes e) = \eta\otimes De + \sum_i \nabla_{e_i}\eta \otimes \sym_D(e^i)e.
$$
Here $\nabla$ is the Levi-Civita connection on $T^*M$ and $\{e_i\}$ is a local orthonormal tangent frame while $\{e^i\}$ is its dual cotangent frame.
This definition is independent of the choice of frame and yields a well-defined first-order differential operator.
The principal symbol of $\Dt$ is given by $\sym_{\Dt}(\xi) = \id\otimes \sym_D(\xi)$.
Thus, $\Dt$ is also a Dirac-type operator.

Define 
\begin{align*}
\gamma\colon T^*M\otimes E\to F, &\quad\gamma(\xi\otimes v) = \sym_D(\xi) v \text{ and} \\
\iota\colon F\to T^*M\otimes E, &\quad\iota(f) = \tfrac{1}{n}\sum_i e^i\otimes \sym_D(e^i)^*f.
\end{align*}
Straightforward computation shows
\begin{align}
\gamma\circ\iota = \id_F, \label{eq.RS.gammaiota}\\
\iota^* = \tfrac{1}{n}\gamma. \label{eq.RS.iota*}
\end{align}
Equation~\eqref{eq.RS.gammaiota} shows that $\iota\circ\gamma\colon T^*M\otimes E\to T^*M\otimes E$ is a projection onto the image of $\iota$.
By Equation~\eqref{eq.RS.iota*}, the projection $\iota\circ\gamma$ is self-adjoint.
Since $\iota$ is injective, the kernel of this projection is the same as the kernel of $\gamma$.
We define 
$$
E^{\nicefrac32} := \ker(\gamma) \subset T^*M\otimes E.
$$
We have the orthogonal decomposition
$$
T^*M\otimes E = \iota(F) \oplus E^{\nicefrac32}.
$$
There is an analogous orthogonal decomposition $T^*M\otimes F = \tilde\iota(E) \oplus F^{\nicefrac32}$ where
\begin{align*}
\tilde\gamma\colon T^*M\otimes F\to E, &\quad\tilde\gamma(\xi\otimes f) = \sym_D(\xi)^* f, \text{ and} \\
\tilde\iota\colon E\to T^*M\otimes F, &\quad\tilde\iota(v) = \tfrac{1}{n}\sum_i e^i\otimes \sym_D(e^i) v \quad\text{ and} \\
F^{\nicefrac32} &= \ker(\tilde\gamma).
\end{align*}
The \emph{Rarita-Schwinger operator} $D_{\nicefrac32}\colon C^\infty(M;E^{\nicefrac32})\to C^\infty(M;F^{\nicefrac32})$ is defined by 
$$
D_{\nicefrac32} = (\id_{T^*M\otimes F} - \tilde\iota\circ\tilde\gamma) \Dt|_{E^{\nicefrac32}}.
$$
The Rarita-Schwinger operator $D_{\nicefrac32}$ is not of Dirac type but in Appendix~\ref{appendix.RS} we show that it is elliptic.
Theorem~\ref{Thm:cherwolf} will show that $D_{\nicefrac{3}{2}}$ and $D_{\nicefrac{3}{2}}^\dagger$ are complete if the Riemannian metric of $M$ is complete.
\end{example}

Imposing boundary conditions amounts to considering extensions of $D_{\min}$ contained in $D_{\max}$.
To understand these extensions, it is necessary to define the boundary trace map on $\dom(D_{\max})$ as well as characterise $\dom(D_{\min})$.
This is provided by the following theorem. 

\begin{theorem}[The trace theorem \cite{BBan}*{Thm.~2.3~(i) and (ii)}]
\label{Thm:CTrace} 
Under the assumptions~\ref{Hyp.StdFirst}--\ref{Hyp.StdLast}, $\Ck[c]{\infty}(M;E)$ is dense in $\dom(D_{\max})$  with respect to the graph norm and the restriction map to the boundary
$$u \mapsto u \rest{\partial M}\colon \Ck[c]{\infty}(M;E) \to \Ck{\infty}(\partial M;E)$$
has a unique bounded extension
$$ u \mapsto u\rest{\partial M}\colon \dom(D_{\max}) \to \SobH{-\frac12}(\partial M;E).$$ 
The kernel of this extension is precisely $\dom(D_{\min})$.
\end{theorem}

For a closed subspace $B\subset \SobH{\frac12}(\dM; E)$ we define 
$$
\dom(D_B) := \set{u \in \dom(D_{\max}) : u \rest \dM \in B} .
$$
The restriction of $D_{\max}$ to $\dom(D_B)$ is denoted by $D_B$.

\section{Verifying the completeness assumption \ref{Hyp.Complete}}
\label{Sec.Complete}

In this section, we provide a useful geometric criterion for completeness of an operator and its formal adjoint.
Conceptually, any geometric operator on a complete Riemannian manifold satisfies completeness. 

\begin{theorem}[\cite{BBan2}*{Thm.~2.1}]\label{Thm:cherwolf}
Assume \ref{Hyp.StdFirst}--\ref{Hyp.Bundles} and let $D\colon \Ck{\infty}(M;E)\to\Ck{\infty}(M;F)$ be a first-order differential operator.
Suppose $C < \infty$ is a constant and $g$ a complete Riemannian metric on $M$ such that the principal symbol of $D$ satisfies
\begin{equation}
  |\sym_D(\xi)| \le C\cdot|\xi|_g
\label{eq:symest}
\end{equation}
for all $\xi\in T^*M$.
Then \ref{Hyp.Complete} holds, i.e., $D$ and $D^\dagger$ are complete.
\end{theorem}

\begin{remark}
Note that there is no assumption here that $\mu$ is induced by the Riemannian metric~$g$.
Although $D^\dagger$ depends on $\mu$, the principal symbol does not.
\end{remark}

\begin{example} 
\label{Ex:DiracComp}
For any Dirac-type operator $D$ we have
$$
\modulus{\sym_{D}(\xi)u}^2 = h\big(\sym_D(\xi)^*\sym_D(\xi)u,u\big) = \modulus{\xi}_g^2\, \modulus{u}^2
$$
and hence
$$
\modulus{\sym_{D}(\xi)} \le |\xi|_g.
$$
\end{example}

\begin{example}
\label{Ex:RSComp}
Let $D_{\nicefrac32}$ be a Rarita-Schwinger operator.
From the computation of the principal symbol of $D_{\nicefrac32}$ in \eqref{eq.RS.symbol}, we see that the estimate~\eqref{eq:symest} holds with $C=1$.
\end{example}

\begin{remark}
A slightly more general version of Theorem~\ref{Thm:cherwolf} can be obtained by replacing the constant $C$ in this theorem by the quantity $C( \mathrm{dist}(p,x))$ where $p \in M$ is a fixed point and  $C\colon [0,\infty)\to\R$ is a positive monotonically increasing continuous function satisfying: 
\begin{equation*}
  \int_0^\infty \frac{dr}{C(r)} = \infty.
\end{equation*}
\end{remark}

\section{Elliptic boundary conditions}
\label{Sec.EBC}

Regularity is a local question and interior regularity is furnished simply from the ellipticity of the operator $D$. 
Given that we have defined the boundary restriction map on $\dom(D_{\max})$, we are able to consider the question of regularity up to the boundary. 

\begin{theorem}[\cite{BBan}*{Thm.~2.4}]
\label{Thm:Reg}
Under~\ref{Hyp.StdFirst}--\ref{Hyp.StdLast}, we have that:
\begin{align}
\label{Eq:MaxDom}
&\dom(D_{\max}) \intersect \SobH[loc]{k+1}(M;E) \notag\\
&\qquad=
\set{ u \in \dom(D_{\max}): Du \in \SobH[loc]{k}(M;E)\text{ and } u\rest{\dM} \in \SobH{k+\frac{1}{2}}(\dM;E)}.
\end{align}
\end{theorem} 

By the Sobolev embedding theorem, we therefore have:
\begin{align*}
\dom(D_{\max}) &\intersect \Ck{\infty}(M;E) \\
&=
\set{ u \in \dom(D_{\max}): Du \in \Ck{\infty}(M;E)\text{ and } u\rest{\dM} \in \Ck{\infty}(\dM;E)}.
\end{align*}

In the next subsection, we will see that the condition $u\rest{\dM} \in \SobH{k+\frac{1}{2}}(\dM;E)$ in \eqref{Eq:MaxDom} can be relaxed, see Theorem~\ref{Thm:Trace}~\ref{Thm:Trace4}.

\subsection{Adapted boundary operators}
To describe elliptic boundary conditions, we require the notion of adapted boundary operators.

\begin{definition}[Adapted boundary operator]
\label{Def.AdOp} 
Assume \ref{Hyp.StdFirst} and \ref{Hyp.Inward}--\ref{Hyp.D}.
We say that a differential operator $A\colon\Ck{\infty}(\dM;E) \to \Ck{\infty}(\dM;E)$ is an \emph{adapted operator} for $D$ if the principal symbol of $A$ satisfies: 
 \begin{equation}
\label{Eq:AdSym}
\sym_{A}(\xi) = \sym_{D}(\tau)^{-1} \circ \sym_{D}(\xi).
\end{equation} 
Here we identify $\xi\in T^*_x\dM$ with its extension to $T_x M$ which satisfies $\xi(T) = 0$.
\end{definition}

\begin{remark}
\label{Rem.AdOp}
The construction of $\sym_A$ and the notion of adapted  operator are still meaningful if we restrict to a two-sided hypersurface $N\subset M$ instead of $\dM$.
\end{remark}

Clearly $A$ is a first-order elliptic differential operator. 
Such an operator always exists. 
Its spectrum is discrete.
The projectors $\chi^{\pm}(A)$ projecting onto the eigenspaces for the eigenvalues with positive or non-positive real part, respectively, exist and act boundedly $\chi^{\pm}(A)\colon \SobH{\alpha}(\dM;E) \to \SobH{\alpha}(\dM;E)$ for all $\alpha \in \R$.

\begin{theorem}[\cite{BBan}*{Thm.~2.3~(iv) and Thm.~2.4}]
\label{Thm:Trace} 
Assume \ref{Hyp.StdFirst}--\ref{Hyp.StdLast}. 
Let  $A$ be an adapted boundary operator for $D$.
Then:  
\begin{enumerate}[label=(\roman*), labelwidth=!, labelindent=2pt, leftmargin=21pt]
\item 
\label{Thm:Trace3}
 For all $u \in \dom(D_{\max})\cap \SobH[loc]{1}(M;E)$ and $v \in \dom((D^\dagger)_{\max})\cap \SobH[loc]{1}(M;F)$,
\begin{equation*}
\label{Eq:MaxDInt}
\inprod{D_{\max} u, v}_{\Lp{2}(M;F)} - \inprod{u, (D^\dagger)_{\max}v}_{\Lp{2}(M;E)} = -\inprod{ u\rest{\dM}, \sym_0^\ast v\rest{\dM}}_{\Lp{2}(\dM;E)}.
\end{equation*}
\item 
\label{Thm:Trace4} 
The boundary regularity can be described in terms of $\chi^+(A)$ by: 
\begin{multline*}
\dom(D_{\max}) \cap \SobH[loc]{k+1}(M;E) \\
	=  \set{u \in \dom(D_{\max}): Du \in \SobH[loc]{k}(M;F)\ \text{and}\ \chi^{+}(A)(u\rest{\dM}) \in \SobH{k + \frac{1}{2}}(\dM;E)}.
\end{multline*}
\end{enumerate}
\end{theorem} 

\begin{example}\label{ex.DiracAdOp}
For a Dirac-type operator $D$ on a Riemannian manifold we choose $T$ to be the inward pointing unit normal vector field and its induced conormal field $\tau$ along $\dM$.
Then $\sym_D(\tau)^*\sym_D(\tau) = |\tau|^2=1$ and hence $\sym_D(\tau)^* = \sym_D(\tau)^{-1}$.
Therefore,  
\begin{align*}
\sym_A(\xi)^*
&=
\sym_D(\xi)^*\circ(\sym_D(\tau)^{-1})^*
=
\sym_D(\xi)^*\circ\sym_D(\tau) \\
&=
-\sym_D(\tau)^*\circ\sym_D(\xi) 
=
-\sym_D(\tau)^{-1}\circ\sym_D(\xi) 
=
-\sym_A(\xi).
\end{align*}
Moreover, $\sym_A$ also satisfies the Clifford relations.
Thus, $A$ is also of Dirac-type and can be chosen to be self-adjoint.
\end{example}

\begin{example}\label{ex.RSAdOp}
The symbol $\sym_A$ defined in \eqref{Eq:AdSym} for the Rarita-Schwinger operator $D=D_{\nicefrac{3}{2}}$ is not skew-symmetric, see \cite{BBan}*{Sec.~3.3}.
Thus the adapted operator $A$ cannot be chosen self-adjoint and is not again a Rarita-Schwinger operator.
\end{example}

\subsection{General theory of elliptic boundary conditions}
We identify a class of ``good'' boundary conditions for which we will obtain regularity up to the boundary.
We start by giving an abstract definition of elliptic boundary conditions.

\begin{definition}[$\infty$-Elliptic boundary condition]
\label{Def.ER}
A closed subspace $B\subset \SobH{\frac12}(\dM;E)$ is called an \emph{$\infty$-elliptic boundary condition} for $D$ if 
$$
B = W_+ \oplus \set{v + gv: v \in V_{-}\cap \SobH{\frac12}(\dM;E)}
$$
where 
\begin{enumerate}[label=(\roman*), labelwidth=!, labelindent=2pt, leftmargin=21pt]
\item  \label{Def.EllBC:First} \label{Def.EllBC:MutualComp}
$W_{\pm}$, $V_{\pm}$ are mutually complementary subspaces of $\Lp{2}(\dM;E)$ such that 
	$$V_\pm \oplus W_\pm = \chi^{\pm}(A) \Lp{2}(\dM;E),$$
\item  \label{Def.EllBC:FiniteDim}
$W_{\pm}$ are finite dimensional with $W_{\pm}, \widetilde{W}_{\pm} := (V_+\oplus V_- \oplus W_\mp)^{\perp,\Lp{2}} \subset \Ck{\infty}(\dM;E)$, and 
\item  \label{Def.EllBC:Last} \label{Def.EllBC:Bddmap}
$g\colon \Lp{2}(\dM;E) \to \Lp{2}(\dM;E)$ is bounded linear with  
\begin{align*} 
&g|_{V_+\oplus W_+\oplus W_-} = 0,\\
&g(V_-)\subset V_+,\\
&g(V_-\cap \SobH{s}(\dM;E)) \subset V_+\cap \SobH{s}(\dM;E)\quad \text{and} \\
&g^\ast(\widetilde{V}_{+} \cap \SobH{s}(\dM;E)) \subset \widetilde{V}_- \cap \SobH{s}(\dM;E) .
\end{align*}
for all $s\ge\frac12$, where $g^*\colon \Lp{2}(\dM;E) \to \Lp{2}(\dM;E)$ is the adjoint map of $g$ and $\widetilde{V}_{\pm} = (V_{\mp} \oplus W_+\oplus W_-)^{\perp,\Lp{2}}$.
\end{enumerate}
\end{definition}

From $g|_{V_+\oplus W_+\oplus W_-} = 0$ it follows that $g^*(\widetilde{V}_+)\subset \widetilde{V}_-$.
If the decomposition  $\Lp{2}(\dM;E) = V_- \oplus W_- \oplus V_+ \oplus W_+$ is orthogonal, then $\widetilde{W}_{\pm}=W_\pm$ and $\widetilde{V}_{\pm}=V_\pm$.

\begin{remark}
If $B\subset \SobH{\frac12}(\dM;E)$ is an $\infty$-elliptic boundary condition for $D$, then $D_B\colon \dom(D_B)\to \Lp{2}(M;F)$ is a closed operator and $\dom\big(\big(D_{B}\big)^*\big)\subset \SobH[loc]{1}(M;F)$.
\end{remark}

\begin{remark}
In \cite{BBan} a weaker notion of elliptic boundary condition was introduced.
Our notion of $\infty$-elliptic boundary condition is equivalent to that of ``$\infty$-regular elliptic boundary condition'' in \cite{BBan}*{Definition~2.11}.
\end{remark}

\begin{remark}
Ellipticity of a boundary condition $B\subset\SobH{\frac12}(\dM;E)$ depends on $D$ but is independent of the choice of adapted boundary operator $A$ as one can see from Corollary~\ref{Cor.InfReg}.
\end{remark}

\begin{example}
\label{Ex.APS}
If we put $W_+=W_-=0$ and $g=0$, then $B=V_-=\chi^-(A)\SobH{\frac12}(\dM;E)$ is an $\infty$-elliptic boundary condition.
Then $B=:B_{\APS}(A)$ is called the \emph{Atiyah-Patodi-Singer boundary condition}.
\end{example}

Elliptic boundary conditions enjoy the best possible regularity properties as outlined in the following theorem.

\begin{theorem}[\cite{BBan}*{Thm.~2.12}]
\label{thm.BoundaryReg}
Assume \ref{Hyp.StdFirst}--\ref{Hyp.StdLast}. 
Let  $A$ be an adapted boundary operator for $D$ and let $B\subset\SobH{\frac12}(\dM;E)$ be an $\infty$-elliptic boundary condition.
Then for all $k\in\N_0$: 
\begin{align*}
\dom(D_B) &\cap \SobH[loc]{k+1}(M;E) \\
&=  
\big\{u \in \dom(D_B): Du \in \SobH[loc]{k}(M;F)\ \text{and}\ u\rest{\dM} \in \SobH{k + \frac{1}{2}}(\dM;E)\big\}.
\end{align*}
In particular,
\begin{align*}
\dom(D_B) &\intersect \Ck{\infty}(M;E)  \\
&=
\big\{ u \in \dom(D_B): Du \in \Ck{\infty}(M;E)\text{ and } u\rest{\dM} \in \Ck{\infty}(\dM;E)\big\}.
\end{align*}
\end{theorem}

To understand the adjoint problem, the following definition will be useful.
Let $\sym_0  := \sym_D(\tau)$ be the principal symbol of $D$ in the conormal direction.
Viewing $B$ as a subspace of $\SobH{-\frac12}(\dM;E)$, we put
$$
B^\dagger := \big\{v \in \SobH{\frac12}(\dM;F) : \inprod{  u, \sym_0^\ast v}_{\SobH{-\frac12} \times \SobH{\frac12}} = 0\quad \forall u \in B\big\}.
$$

\begin{example}
\label{Ex.APSdagger}
If $A$ is self-adjoint and $\sym_0$ anti-commutes with $A$, then $B_{\APS}(A)=B_{\APS}(A)^\dagger\oplus\ker(A)$.
In particular, $B_{\APS}(A)=B_{\APS}(A)^\dagger$ if and only if $\ker(A)=0$.
\end{example}

\begin{theorem}[\cite{BBan}*{Prop.~8.2}]
Assume \ref{Hyp.StdFirst}--\ref{Hyp.StdLast}. 
Let $B$ be an $\infty$-elliptic boundary condition for $D$.
Then $B^\dagger$ is an $\infty$-elliptic boundary condition for $D^\dagger$ and the adjoint operator of $D_B$ is given by
$$
(D_B)^* = D^\dagger_{B^\dagger}.
$$
\end{theorem}

\begin{example}
If $D$ is formally self-adjoint, $D^\dagger=D$, and $A$ is as in Example~\ref{Ex.APSdagger} with $\ker(A)=0$, then $D_{B_{\APS}(A)}$ is self-adjoint.
\end{example}

When a boundary condition $B$ is $\infty$-elliptic and as described above, then the adjoint boundary is given by
$$
\sym_0^\ast B^\dagger = \widetilde{W}_- \oplus \big\{ u - g^\ast u: u \in \widetilde{V}_+ \cap \SobH{\frac12}(\dM;E)\big\} .
$$

Also, we note there are some other important characterisations of $\infty$-elliptic boundary conditions, particularly in the language of Fredholm pairs.
These are treated in depth in \cite{BBan}.

\subsection{Relation to the classical treatment of boundary conditions} 
Traditionally, boundary conditions were treated by pseudo-differential methods.
We now show how these classical considerations can be captured through our setup.

\begin{definition}[Pseudo-local and local boundary conditions]
If $P$ is a classical pseudo-differential projector of order zero, then 
$$ B := P \SobH{\frac12}(\dM;E)$$ 
is called a \emph{pseudo-local} boundary condition.

If $P$ arises out of a fibrewise smooth projection to a subbundle $E'$, then it is a \emph{local} boundary condition. 
\end{definition} 

Note that if $P$ defines a local boundary condition, i.e., it is a fibrewise smooth projection to a subbundle $E'$, then $B = \SobH{\frac12}(\dM;E')$.

It is especially useful to know when a pseudo-local boundary condition is $\infty$-elliptic as characterised in the following theorem. 
\begin{theorem}[\cite{BBan}*{Thm.~2.15}]
\label{Thm:PL}
Assume \ref{Hyp.StdFirst}--\ref{Hyp.StdLast}. 
For a pseudo-local boundary condition $B = P\, \SobH{\frac{1}{2}}(\dM;E)$, the following are equivalent: 
\begin{enumerate}[(i)]
\item 
\label{Thm:PL1}
$B$ is $\infty$-elliptic.
\item 
\label{Thm:PL2}
For some/every invertible bisectorial adapted boundary operator $A$,  
	$$P - \chi^{+}(A)\colon \Lp{2}(\dM;E) \to \Lp{2}(\dM;E)$$
	is a Fredholm operator.
\item 
\label{Thm:PL3}
For some/every invertible bisectorial adapted boundary operator $A$, 
	$$P - \chi^{+}(A)\colon \Lp{2}(\dM;E) \to \Lp{2}(\dM;E)$$
	is elliptic.
\item 
\label{Thm:PL4} 
For some/every adapted boundary operator $A$, and for every $\xi \in T_x^\ast \dM\setminus \set{0}$, $x \in \dM$, the principal symbol $\sym_{P}(x,\xi)\colon E_x \to E_x$ restricts to an isomorphism from the sum of the generalised eigenspaces of $\imath\sym_{A}(x,\xi)$ to the eigenvalues with negative real part onto the image $\sym_{P}(x,\xi)(E_x)$ and, similarly, $\sym_{P^\ast}(x,\xi)$ restricts to an isomorphism from the sum of the generalised eigenspaces of $\imath\sym_{A^*}(x,\xi)$ to the eigenvalues with negative real part onto $\sym_{P^*}(x,\xi)(E_x)$.
\end{enumerate} 
\end{theorem}

The last condition, Theorem~\ref{Thm:PL}~\ref{Thm:PL4}, is named after Lopatinsky and Schapiro.
\begin{corollary}
\label{cor.local}
If $E\rest{\dM} = E' \oplus E''$ is a smooth decomposition into subbundles and $\sym_{A}(\xi)$ interchanges $E'$ and $E''$ for every $\xi \in T^\ast\dM$, then $B' := \SobH{\frac12}(\dM;E')$ and $B'' := \SobH{\frac12}(\dM;E'')$ are both $\infty$-elliptic boundary conditions for $D$. 
\end{corollary} 

\begin{example}
Let $E=F=\bigoplus_{k=0}^n\Lambda^k T^*_{\C}M$ be the complexification of the bundle of differential forms over a complete $n$-dimensional Riemannian manifold.
Let $d$ be the exterior differential and put $D:=d+d^\dagger$.
Then $D$ is of Dirac type.

As before, let $T$ be the interior unit normal vector field along the boundary $\dM$ and $\tau$ the associated unit conormal one-form.
For $0\le j\le n$ we have a canonical identification 
\[
\Lambda^jT^*_{\C}M
= \big(\Lambda^jT^*_{\C}\dM\big) \oplus \big(\tau\wedge\Lambda^{j-1}T^*_{\C}\dM\big), 
\quad
\phi = \phi^{\tan} + \tau\wedge \phi^\mathrm{nor}.
\]
The local boundary condition corresponding to the subbundle
$$
E':=\bigoplus_{k=0}^{n-1}\Lambda^k T^*_{\C}\dM \subset E_{|\dM}
$$ 
is called the {\em absolute boundary condition},
\[
B_\mathrm{abs} 
= 
\{\phi\in \SobH{\frac12}(\dM;E) : \phi^\mathrm{nor} =0 \} ,
\]
while $E'':=\tau\wedge\bigoplus_{k=0}^{n-1}\Lambda^k T^*_{\C}\dM \subset E_{|\dM}$ yields the {\em relative boundary condition},
\[
B_\mathrm{rel}
=
\{\phi\in \SobH{\frac12}(\dM;E) : \phi^{\tan} =0 \} .
\]
The normal principal symbol of $D$ is given by $\sym_D(\tau)\omega = \tau\wedge\omega + T\lrcorner\,\omega$ and interchanges the subbundles $E'$ and $E''$ while the tangential principal symbol $\sym_D(\xi)$ preserves the splitting (for $\xi\in T^*\dM$).
Therefore, the principal symbol $\sym_A(\xi)$ of the adapted boundary operator $A$ interchanges $E'$ and $E''$.
Corollary~\ref{cor.local} implies that both $B_\mathrm{abs}$ and $B_\mathrm{rel}$ are $\infty$-elliptic boundary conditions for $D$.
These boundary conditions are important in geometry because the solutions of the homogeneous boundary value problems represent elements of the absolute and relative cohomology groups of $M$, respectively.
\end{example}

\section{Fredholmness}
\label{Sec.Fred}

To study the Fredholm property of a boundary value problem, we recall the following definition. 
\begin{definition}
\label{Def.Coercive}
The operator  $D$ is said to be \emph{coercive at infinity} if there exists $C > 0$ and a compact $ K\subset M$ such that 
$$
\norm{u}_{\Lp{2}( M; E)}\leq C\,\norm{D u}_{\Lp{2}( M; F)}
$$
for all $u\in\Ck{\infty}( M; E)$  such that $\spt u \subset M\setminus  K$.
\end{definition}

If $M$ itself is compact, then we can choose $K=M$ and $D$ is automatically coercive at infinity.

Elliptically regular boundary conditions give rise to Fredholm operators when the underlying operator $D$ and its formal adjoint $D^\dagger$ are coercive at infinity. 

\begin{theorem}[Fredholmness \cite{BBan}*{Thm.~2.19}]
\label{Thm:Fredholm}
Assume \ref{Hyp.StdFirst}--\ref{Hyp.StdLast}. 
Let $D$ and $D^\dagger$ be coercive at infinity and let $B$ be an $\infty$-elliptic boundary condition for $D$.
Then, the following hold:
\begin{enumerate}[(i)]
\item  \label{Thm:Fredholm:1}
$D_B$ is a Fredholm operator and 
	$$ \Ind(D_B) = \dim \ker D_B - \dim \ker D^\dagger_{B^\dagger} \in \In.$$
\item \label{Thm:Fredholm:2} 
Let $C$  be  a closed complementary subspace to $B$ in $\SobH{\frac12}(\dM;E)$ with an associated projection $Q\colon\SobH{\frac12}(\dM;E)\to \SobH{\frac12}(\dM;E)$ with kernel $B$ and image $C$.
Then
$$
L\colon \dom(D_{\max})\cap \SobH[loc]{1}(M;E) \to \Lp{2}(M;F) \oplus C,\quad Lu := \big(D_{\max}u, Q(u\rest{\dM})\big)
$$
is a Fredholm operator with the same index as $D_B$.
\item  \label{Thm:Fredholm:3}
If $B' \subset B$ is another $\infty$-elliptic boundary condition, then $\dim \Big(\faktor{B}{B'}\Big) < \infty$ and 
	$$\Ind(D_B) = \Ind (D_{B'}) + \dim\Big(\faktor{B}{B'}\Big).$$
\end{enumerate} 
\end{theorem} 

\begin{example}
Let $E=F$ be the spinor bundle over a complete Riemannian spin manifold $M$ and $D$ the spinorial Dirac operator.
Then, if the scalar curvature of $M$ is uniformly positive outside a compact subset of $M$, the Dirac operator is coercive at infinity by the Lichnerowicz formula \cite{Lic}*{Eq.~(7)}.
Hence, if we impose $\infty$-elliptic boundary conditions $B$ such as the APS condition, then $D_B$ is Fredholm.
\end{example}

Let $B$ be an $\infty$-elliptic boundary condition. 
By considering a parameter $s \in [0,1]$ and defining 
$$B_s := W_+ \oplus \set{ v + sgv: v \in V_{-}\cap \SobH{\frac12}(\dM;E)},$$ 
we obtain a continuous family of boundary conditions $(B_s)_{s \in [0,1]}$.
This results in a continuous deformation $s \mapsto D_{B_{s}}$.
By deformation invariance of the index of Fredholm operators, we have for all $s \in [0,1]$
$$\Ind(D_{B}) = \Ind(D_{B_s}) = \Ind(D_{B_0}).$$
Since $B_0 = W_+ \oplus B_{\APS}(A)$, the index calculation of a general $\infty$-elliptic boundary condition can be reduced that for a finite-dimensional modification of the APS condition. 
This results in the formula 
\begin{equation}
\label{Eq:Deform}
\Ind(D_{B}) = \Ind (D_{B_{\APS}(A)}) + \dim W_+ - \dim W_-.
\end{equation}

\section{The matching boundary condition and relative index theory}
\label{Sec.Match} 

In this final section, we apply the theory to derive a relative index theorem.
This requires the introduction of a new $\infty$-elliptic boundary condition, the matching condition, see Definition~\ref{Def.Matching} below.

For the remainder of this subsection, let $M^\prime$ be a boundaryless manifold. 
Let $ N\subset M^{\prime}$ be a two-sided compact hypersurface in $ M^{\prime}$ (i.e.\ $ N$ has a trivial normal bundle).
Then by ``cutting along $ N$'', we obtain the manifold with boundary
\[
 M:=( M^{\prime}\setminus N) \cup( N_{1}\sqcup N_{2}) ,
\]
where $ N_{1}= N$, $ N_{2}=- N$ (i.e.\ with opposite orientation) and with $\partial M= N_{1}\sqcup N_{2}$.

We obtain a natural smooth map $\Xi\colon M\to M'$ which is a diffeomorphism onto $M'\setminus N$ on the interior of $M$ and maps $N_i$ diffeomorphically onto $N$.

A density $\mu^{\prime}$ on $ M^{\prime}$ and bundles $ E', F'\to M^{\prime}$ can be pulled back along $\Xi$ to yield corresponding objects $\mu$, $E$, $F$ on $M$.
Similarly, an operator $D'\colon\Ck{\infty}( M^{\prime}, E')\to\Ck{\infty}( M^{\prime}, F')$ induces an operator $D\colon\Ck{\infty}( M, E)\to\Ck{\infty}( M, F)$.

\begin{definition}[Matching condition]
\label{Def.Matching}
The subspace 
\[
\Bmatch:=\set{ (u,u) \in \SobH{\frac12}( N_1;  E) \oplus \SobH{\frac12}( N_2;  E) : u\in\SobH{\frac{1}{2}}( N; E) }  \subset  \SobH{\frac12}(\dM;E)
\]
is called the \emph{matching condition} where we identify $N_1$ and $N_2$ with $N$.
\end{definition}

We choose an adapted operator $A_N$ for $D'$ on the hypersurface $N$, see Remark~\ref{Rem.AdOp}.
Replacing $A_N$ by $A_N+r\,\id$ for some $r\in\R$ if necessary, we can assume that $A_N$ is invertible and bisectorial.
Now $A := A_N \oplus (-A_N)$ is an invertible bisectorial adapted boundary operator for $D$ on $\dM = N_1 \sqcup N_2$.
Upon identifying $N_1$ and $N_2$ with $N$, we observe that 
\begin{align*} 
B_{\APS}(A) 
&= \chi^-(A_N) \SobH{\frac12}(N_1;E) \oplus \chi^-(-A_N) \SobH{\frac12}(N_2;E)  \\ 
&= \chi^-(A_N) \SobH{\frac12}(N;E) \oplus \chi^+(A_N) \SobH{\frac12}(N;E).
\end{align*} 
Putting 
\begin{align*}
V_- 
&:=
\chi^-(A_N) \Lp{2}(N;E) \oplus \chi^+(A_N) \Lp{2}(N;E), \\
V_+ 
&:=
\chi^+(A_N) \Lp{2}(N;E) \oplus \chi^-(A_N) \Lp{2}(N;E), \\
W_-&:= W_+ := 0,\\
g&\colon \Lp{2}(\dM;E)\to \Lp{2}(\dM;E),\\ 
g|_{V_-}&\colon V_-\to V_+, \quad (u,v) \mapsto (v,u), \quad\text{ and }\quad g|_{V_+}=0,
\end{align*}
we find 
\begin{align*}
\Bmatch = \{x+gx : x\in V_-\cap \SobH{\frac12}(\dM;E)\} .
\end{align*}
If we assume that $D'$ is coercive at infinity, then so is $D$ on $M$.
Hence, we get Fredholm operators and Equation~\eqref{Eq:Deform} yields
\begin{equation} 
\label{Eq:MatchDef} 
\Ind(D') = \Ind(D_{\Bmatch}) = \Ind(D_{B_{\APS}(A)}).
\end{equation} 

\begin{theorem}[Relative index theorem \cite{BRelIndex}*{Theorem 1.1}]
\label{Thm:RelInd}
Let $(M_1,\mu_1,E_1,F_1,D_1)$ and $(M_2,\mu_2,E_2,F_2,D_2)$ satisfy Assumptions~\ref{Hyp.StdFirst}--\ref{Hyp.StdLast} with $\dM_1=\dM_2=\emptyset$.
Assume there exist compact subsets $K_1\subset M_1$ and $K_2\subset M_2$ such that $\mu_1=\mu_2$, $E_1=E_2$, $F_1=F_2$ and $D_1=D_2$ on $M_1\setminus K_1=M_2\setminus K_2$.

Then $D_{1}$ is Fredholm  if and only if  $D_{2}$ is Fredholm and in that case
\begin{equation}
\Ind(D_{1}) -\Ind(D_{2}) =\int_{K_{1}}\alpha_{0,D_{1}} -\int_{K_{2}}\alpha_{0,D_{2}},
\label{eq.RelInd}
\end{equation}
where $\alpha_{0,D_{i}}$ is the \emph{local} index density of $D_{i}$.  
\end{theorem}

\begin{proof}[Sketch of proof]
The operator $D_{1}$ is Fredholm if and only if $D_{1}$ and $D_{1}^\dagger$ are coercive at infinity.
Extending the compact set where $D_1, D_1^\dagger$ are coercive to include $K_1$, we see that $D_2 = D_1, D_2^\dagger = D_1^\dagger$ on $M_2 \setminus K_2$ satisfies the coercivity property outside of $K_2$.
This is equivalent to the Fredholmness of $D_2$. 

We take $N \subset M_1 \setminus K_1 = M_2 \setminus K_2$ a smooth compact $2$-sided hypersurface which decomposes $M_1$ and $M_2$ such that $M_1=M_1'\sqcup_N M_1''$  and $M_2=M_2'\sqcup_N M_2''$, respectively.
Here $M_i'$ is compact and contains $K_i$ and $M_1''=M_2''$.

Let $\tilde{D}_1 = D_1' \oplus D_1''$ and $\tilde{D}_2 = D_2'\oplus D_2''$ be the induced operators on $\tilde{M}_1:= M_1'\sqcup M_1''$ and $\tilde{M}_2:= M_2'\sqcup M_2''$.
Both $\tilde{M}_1$ and $\tilde{M}_2$ have the same boundary $N\sqcup (-N)$ on which we impose the APS-boundary condition $B_{\APS}(A)=B'\oplus B''$ where $B' := \chi^-(A_N) \SobH{\frac12}(N;E)$ and $B'' = \chi^-(-A_N) \SobH{\frac12}(N;E)$. 
Equation~\eqref{Eq:MatchDef} yields
\begin{align*}
\Ind(D_1) &= \Ind(\tilde{D}_{1, \Bmatch}) = \Ind(\tilde{D}_{1, B_{\APS}(A)}) = \Ind(D_{1, B'}') + \Ind(D_{1, B''}''),   \\
\Ind(D_2) &= \Ind(\tilde{D}_{2, \Bmatch}) = \Ind(\tilde{D}_{2, B_{\APS}(A)}) = \Ind(D_{2, B'}') + \Ind(D_{2, B''}'').
\end{align*}
Since $D_1'' = D_2''$,
\[
\Ind(D_1) - \Ind(D_2) =  \Ind(\tilde{D}_{1, B'}') - \Ind(\tilde{D}_{2, B'}').
\]
We choose $(M_3,\mu_3,E_3,F_3,D_3)$ where $M_3$ is compact and has boundary $-N$ such all data match smoothly on $M_1'\sqcup_N M_3$.
Since $M_1'$ and $M_2'$ and their data agree on a neighbourhood of $N$, the data also match smoothly on $M_2'\sqcup_N M_3$.
Arguing as above yields
\[
\Ind(D_1'\oplus D_3) - \Ind(D_2'\oplus D_3) =  \Ind(\tilde{D}_{1, B'}') - \Ind(\tilde{D}_{2, B'}').
\]
Since $M_i'\sqcup_N M_3$ is closed, the Atiyah-Singer index theorem gives us
\begin{align*}
\Ind(D_i'\oplus D_3) 
&= 
\int_{M_i'\sqcup_N M_3} \alpha_0(D_i'\oplus D_3) \\
&=
\int_{M_i'} \alpha_0(D_i') + \int_{M_3} \alpha_0(D_3) \\
&=
\int_{M_i'} \alpha_0(D_i) + \int_{M_3} \alpha_0(D_3) .
\end{align*}
Therefore,
\begin{equation*}
\Ind(D_1) - \Ind(D_2)
=
\int_{M_1'} \alpha_0(D_1) - \int_{M_2'} \alpha_0(D_2)
=
\int_{K_1} \alpha_0(D_1) - \int_{K_2} \alpha_0(D_2) .
\qedhere
\end{equation*}
\end{proof}

\begin{remark}
In Theorem~\ref{Thm:RelInd} we can allow $M_1$ and $M_2$ to have non-empty compact boundary, equipped with $\infty$-elliptic boundary conditions for $D_1$ and $D_2$, respectively.
The theorem still holds but we get additional boundary terms in the index formula~\eqref{eq.RelInd}.
\end{remark}

\section{Note on higher regularity}
\label{Sec.HighReg}

Regularity of higher regularity of sections in the maximal domain with respect to a boundary condition depends on the boundary condition itself.
In this subsection, we show how to characterise $k$-regularity of a boundary condition $B$ (in the sense of \cite{BBan}*{Definition~2.11}) via the boundary trace map.
We begin with the following technical lemma which aids the proof of this characterisation as given in Theorem~\ref{Thm.HigherReg}. 

\begin{lemma}
\label{Lem.Density}
Let $N$ be a compact manifold satisfying assumptions \ref{Hyp.StdFirst}--\ref{Hyp.StdLast} with $D\colon \Ck{\infty}(N;E) \to \Ck{\infty}(N;F)$ a first-order elliptic operator.
Let $j\in\N$.
If $u \in \dom(D_{\max})$ with $u \in \SobH{j}(N;E)$ and  $D_{\max}u \in \SobH{j}(N;F)$, then there exists a sequence $u_n \in \Ck{\infty}(N;E)$ such that $u_n \to u$ in $\SobH{j}(N;E)$ and $D_{\max} u_n \to D_{\max} u$ in $\SobH{j}(N;F)$.
\end{lemma}
\begin{proof}
Let $D_{\max,j}$ denote $D_{\max}$ acting as a densely-defined operator $\SobH{s}(N;E) \to \SobH{s}(N;F)$.
By \cite{BGS}*{Theorem~6.1}, we can write $\dom(D_{\max,j}) = \SobH{j+1}(N;F) + \Poiss_s \SobH{j-\frac12}(\partial N;E)$, where $\Poiss_s: \SobH{j-\frac12}(\partial N;E) \to \dom(D_{\max,s})$ is the Poisson operator. 
Hence, we can write $u = v + Kw$.
Now, choose $v_n \to v$ in $\SobH{j+1}(N;E)$ with $v_n \in \Ck{\infty}(N;E)$ and $w_n \to w$ in $\SobH{j-\frac12}(\partial N;E)$ with $w_n \in \Ck{\infty}(\partial N;E)$.
Define $u_n := v_n + \Poiss_s w_n$
We have $u_n \in \Ck{\infty}(N;E)$ since $\Poiss_s$ is the Poisson operator and maps smooth sections to smooth sections.
With this, 
\begin{align*}
\norm{(u_n - u)}_{D_{\max,j}} 
&\leq 
\norm{v_n - v}_{D_{\max,j}} + \norm{\Poiss_s (w_n - w)}_{D_{\max,s}} \\ 
\leq
&C_{1,j}  \norm{v_n - v}_{\SobH{j+1}(N;E)} + C_{2,j}\norm{w_n - w}_{\SobH{j-\frac12}(\partial N; E)}
\to 0
\end{align*}
as $n \to \infty$ where $C_{1,j}, C_{2,j} < \infty$ are mapping constants dependent on $j$.
This yields $u_n \to u$ in $\SobH{j}(N;E)$ and $D_{\max}u_n \to D_{\max}u$ in $\SobH{j}(N;F)$.
\end{proof} 

\begin{theorem}
\label{Thm.HigherReg}
Assume \ref{Hyp.StdFirst}--\ref{Hyp.StdLast}.
Let $B\subset \SobH{\frac12}(\dM;E)$ be a closed subspace.
Then for all $k\in\N_0$ the following are equivalent:
\begin{enumerate}
\item 
\label{It.HighReg1} 
$B$ is a $k$-regular elliptic boundary condition for $D$ (in the sense of \cite{BBan}*{Definition~2.11} w.r.t.\ an adapted boundary operator $A$);
\item
\label{It.HighReg2}
for all $j=0,\dots,k-1$ we have:
\begin{align}
\label{Eq.InfRegChar1} 
\dom&(D_{B,\max}) \cap \SobH[loc]{j+1}(M;E) \nonumber  \\
&= 
\set{u\in\dom(D_{B,\max}): Du\in\SobH[loc]{j}(M;F)\text{ and }u\rest{\dM}\in\SobH{j+\frac12}(\dM;E)} ,\\ 
\label{Eq.InfRegChar2} 
\dom&(D^\dagger_{B^\dagger,\max})\cap \SobH[loc]{j+1}(M;F) \nonumber  \\
&=   
\set{u\in\dom(D^\dagger_{B^\dagger,\max}): D^\dagger u\in\SobH[loc]{j}(M;E)\text{ and }u\rest{\dM}\in\SobH{j+\frac12}(\dM;F)} .
\end{align} 
\end{enumerate}
\end{theorem}
\begin{proof}
The implication ``\eqref{It.HighReg1}$\implies$\eqref{It.HighReg2}'' is proved in \cite{BBan}*{Theorem~2.12}.

To prove ``\eqref{It.HighReg2}$\implies$\eqref{It.HighReg1}'', we first reduce the question to that of a ``model'' problem.
For that, let $Z_\rho = [0,\rho) \times \dM$ for $\rho \in (0, \infty]$, which has boundary $\partial Z_{\rho} = \dM$.

Set $\rho = \infty$ and let $D' = (\frac{\partial}{\partial t} + A)\colon \Ck{ \infty}(Z_\infty;E) \to \Ck{ \infty}(Z_\infty;E)$ be the model operator. 
Here $t$ denotes the coordinate on $[0,\infty)$.

The standard setup \ref{Hyp.StdFirst}--\ref{Hyp.StdLast} is satisfied for the manifold $Z_\infty$ with the measure $|dt|\otimes\nu$, the vector field $\frac{\partial}{\partial t}$ along $\dM$, the bundles obtained by pulling back the restrictions of $E$ and $F$ to $\dM$, and the operator $D'$.
Here $\nu$ is the measure induced by $T$ and $\mu$ on $\dM$.

The operator $D'_{B,\max}$ is the extension with $$\dom(D_{B,\max}') = \set{u \in\dom(D'_{\max}): u \rest{\dM}\in B}.$$
We show that \eqref{Eq.InfRegChar1} implies the corresponding statement for $Z_\infty$, i.e., for all $j=0,\dots,k-1$ we have: 
\begin{equation} 
\begin{aligned}
\label{Eq.InfRegChar3} 
\dom&(D_{B,\max}')\cap \SobH[loc]{j+1}(Z_{\infty};E) \\
&= 
\set{w\in\dom(D_{B,\max}'): D'w\in\SobH[loc]{j}(Z_{\infty};E)\text{ and }w\rest{\dM}\in\SobH{j+\frac12}(\dM;E)}.
\end{aligned}
\end{equation} 

We only need to show the inclusion ``$\supset$''.
Fix $j \in \set{0,\dots,k-1}$ and $w\in\dom(D_{B,\max}')$ with  $D'w\in\SobH[loc]{j}(Z_{\infty};E)$ and $w\rest{\dM}\in\SobH{j+\frac12}(\dM;E)$.
Without loss of generality, we can inductively assume that $w \in \SobH[loc]{j}(Z_{\infty};E)$.
By Lemma~2.4 in \cite{BB12}, there exists $\rho_0\in (0,\infty)$ and an open neighbourhood $U_{\rho_0}$ of $\dM$ in $M$ together with a diffeomorphism $U_{\rho_0}\to Z_{\rho_0}$ which preserves $\dM$ pointwise and identifies $T$ with $\frac{\partial}{\partial t}$, $\tau$ with $dt$, and the measure $\mu$ with $|dt|\otimes \nu$.

Fix $\theta < \rho_0$ to be chosen later. 
Let $\eta \in \Ck[c]{\infty}(Z_{\infty},\R)$ such that $\eta = 1$ on $Z_{\frac{\theta}{2}}$ and $\eta = 0$ outside of $Z_{\frac{3\theta}{4}}$.
Clearly $D'\big((1 - \eta) w\big)\in\SobH[loc]{j}(Z_{\infty};E)$ and hence $(1-\eta) w\in\SobH[loc]{j+1}(Z_{\infty};E)$ by interior elliptic regularity for $D'$.

To show $\eta w \in \SobH[loc]{j+1}(Z_{\infty};E)$, we define $u\in\SobH[loc]{j}(M;E)$ by putting $u = \eta w$ on $U_{\rho_0}$ under the identification with $Z_{\rho_0}$ and $u=0$ outside of $U_{\rho_0}$.
From $u\rest{\dM} = w\rest{\dM} \in B$ we see that $u \in \dom(D_B)$. 
In order to invoke  \eqref{Eq.InfRegChar1}, we show that $Du \in\SobH[loc]{j}(M;F)$.
To that end, note that from \cite{BBan}*{Equation~(39)}, for any $\epsilon > 0$, we have a $\rho$ such that 
$$ 
\norm{Dx}_{\SobH{j}(U_{\rho})} \leq \epsilon \norm{x}_{\SobH{j+1}(U_{\rho})} + \norm{D'x}_{\SobH{j}(U_{\rho})} + \norm{\sym_0 R_0 x}_{\SobH{j}(U_{\rho})}
$$
for all $x \in \Ck{\infty}(U_{\rho})$ where $\sym_0 R_0$ is a pseudo-differential operator of order zero.

Choose an initial $\epsilon := \epsilon_1 = 1$ and  let $\rho_1$ be the guaranteed parameter.
Since $U_{\rho_1}$ is precompact, \eqref{Eq.InfRegChar1} yields $C_{j,1} > 0$ dependent on $j$ and $\rho_1$ such that
$$
\norm{x}_{\SobH{j+1}(U_{\rho_1})}  \leq C_{j,1} \norm{Dx}_{\SobH{j}(U_{\rho})}
$$
for all $x \in \Ck[c]{\infty}(U_{\rho_1};E)$.
Now, choose $\epsilon := \epsilon_2 =  \nicefrac{1}{2C_{j,1}}$ and let $\rho_2 = \min\set{1,\rho_1}$ be the guaranteed parameter. 
For $x \in \Ck[c]{\infty}(U_{\rho_2};E)$, extending it by $0$ outside of $\rho_2$ since $\spt x \subset U_{\rho_2}$, 
\begin{align*}  
\norm{Dx}_{\SobH{j}(U_{\rho_2})} 
&\leq  \frac{1}{2C_{j,1}} \norm{x}_{\SobH{j}(U_{\rho_2})}  + \norm{D'x}_{\SobH{j}(U_{\rho_2})} + \norm{\sym_0 R_2 x}_{\SobH{j}(U_{\rho_2})} \\ 
&\leq  \frac{1}{2C_{j,1}} C_{j,1} \norm{Dx}_{\SobH{j}(U_{\rho_1})}  + \norm{D'x}_{\SobH{j}(U_{\rho_2})} + \norm{\sym_0 R_2 x}_{\SobH{j}(U_{\rho_2})} \\ 
&\leq  \frac{1}{2} \norm{Dx}_{\SobH{j}(U_{\rho_2})}  + \norm{D'x}_{\SobH{j}(U_{\rho_2})} + \norm{\sym_0 R_2 x}_{\SobH{j}(U_{\rho_2})},
\end{align*}
where the last line follows since $\spt x \subset U_{\rho_2} \subset U_{\rho_1}$. 
Hence, 
$$ 
\frac12 \norm{Dx}_{\SobH{j}(U_{\rho_2})} \leq \norm{D'x}_{\SobH{j}(U_{\rho_2})} + \norm{\sym_0 R_2 x}_{\SobH{j}(U_{\rho_2})}
$$
for all $x \in \Ck[c]{\infty}(U_{\rho_2};E)$ and therefore, 
\begin{equation} 
\label{Eq.HjNormEst}
\norm{Dx}_{\SobH{j}(U_{\rho_2})} \lesssim \norm{D'x}_{\SobH{j}(U_{\rho_2})} + \norm{x}_{\SobH{j}(U_{\rho_2})}.
\end{equation}

Now, let $y \in \dom(D_{\max}') \cap \SobH[0]{j}(U_{\rho_2};E)$.
Define $N := U_{\rho_2} \cup_{\rho_2 \times \dM} (- U_{\rho_2})$ and let $D''$ be the elliptic first-order differential operator on $N$ such that $D'' = D'$ on $U_{\rho_2}$.
By Lemma~\ref{Lem.Density}, there exists a sequence $y_n \in \Ck{\infty}(N;E)$ such that $y_n \to y$ in $\SobH{j}(N;E)$.
Without loss of generality, by using a cutoff, we can assume that $\spt y_n \subset U_{\rho_2}$.
By Equation~\eqref{Eq.HjNormEst}, we obtain that $\norm{D(y_k - y_l)}_{\SobH{j}(U_{\rho_2})} + \norm{y_k - y_k}_{\SobH{j}(U_{\rho_2})} \to 0$.
Therefore, $y \in \dom(D_{\max}) \cap \SobH[0]{j}(U_{\rho_2};E)$ and $Dy \in \SobH{j}(U_{\rho_2};F)$.

Setting $\theta = \rho_2$ for this latter choice of $\epsilon = \epsilon_2$ and setting $x = u$, we obtain $Du \in\SobH{j}(U_{\rho};F) \subset \SobH[loc]{j}(M;F)$.
Hence, \eqref{Eq.InfRegChar1} yields $u \in \SobH[loc]{j+1}(M;F)$ which yields $\eta w \in \SobH[loc]{j+1}(Z_{\infty};E)$.
Therefore, $w = (1 - \eta) w + \eta w \in \SobH[loc]{j+1}(Z_{\infty};E)$.
This concludes the proof of \eqref{Eq.InfRegChar3}.

Next, we show that \eqref{Eq.InfRegChar3} yields that $B$ is $k$-semiregular.
From $j = 0$, \eqref{Eq.InfRegChar1} and \eqref{Eq.InfRegChar2} yield that $B$ is elliptic in the sense of \cite{BBan}.
Therefore, we have $B = W_+ \oplus \set{v + gv\colon V_- \in \SobH{\frac12}(\dM;E)}$ where $g(V_- \cap \SobH{\frac12}(\dM;E)) \subset \SobH{\frac12}(\dM;E)$.
We show further that $g (V_- \cap \SobH{j+\frac12}(\dM;E)) \subset \SobH{j+\frac12}(\dM;E)$.

Fix $v_j \in V_- \cap \SobH{j+\frac12}(\dM;E)$.
Let $\ext (v_j + g v_j)(t) := \exp(-t \modulus{A})(v_j + g v_j)$ on $Z_{\infty}$.
Now, 
\begin{equation} 
\label{Eq.ModSol}
D' \ext (v_j + gv_j) = -2\modulus{A} \exp(-t \modulus{A})v_j = -2 \modulus{A}^{\frac12} \exp(-t\modulus{A}) \modulus{A}^{\frac12} v_j
\end{equation}  
since $g v_j \in \chi^+(A)\Lp{2}(\dM;E)$.
Therefore, for $l \leq j$, 
$$
\partial_t^l D' \ext(v_j + gv_j) = (-1)^{l+1} 2 \modulus{A}^{\frac12} \exp(-t \modulus{A}) \modulus{A}^{l + \frac12} v_j
$$
and
$$
\norm{A}^l D' \ext(v_j + gv_j)  = -2 \modulus{A}^{\frac12} \exp(-t \modulus{A}) \modulus{A}^{l + \frac12}.
$$
Furthermore, from the fact that $\modulus{A}$ has a \Hinfty-functional calculus, 
\begin{align*} 
\int_{0}^\infty \| \modulus{A}^{\frac12} \exp(-t \modulus{A}) &\modulus{A}^{l + \frac12} v_j\|_{\Lp{2}(\dM)}^2\ dt \\
&=
\int_{0}^\infty \norm{ t\modulus{A}^{\frac12} \exp(-t \modulus{A}) \modulus{A}^{l + \frac12} v_j}_{\Lp{2}(\dM)}^2\ \frac{dt}{t} \\
&= 
\int_{0}^\infty \norm{ t\modulus{A}^{\frac12} \exp(-t \modulus{A}) \modulus{A}^{l + \frac12} v_j}_{\Lp{2}(\dM)}^2\ \frac{dt}{t} \\
&\lesssim 
\norm{\modulus{A}^{l + \frac12} v_j}_{\Lp{2}(\dM)}^2
\simeq \norm{v_j}_{\SobH{l+\frac12}(\dM)}^2 .
\end{align*}
Therefore, 
\begin{align*}
\norm{D' \ext(v_j + gv_j}_{\SobH{j}(Z_{\infty})} 
&\simeq 
\sum_{l=0}^{j} \norm{ \partial_t^l D' \ext(v_j + gv_j)}_{\Lp{2}(Z_{\infty})} 
+ \norm{ \modulus{A}^l D' \ext(v_j + gv_j)}_{\Lp{2}(Z_{\infty})} \\
&\lesssim 
\norm{v_j}_{\SobH{j+\frac12}(\dM)}.
\end{align*}
Similarly, for $w_+ \in W_+$, a similar application yields that $W_+ \subset \SobH{j+\frac12}(\dM)$.
Since this is true for all $j = 0, \dots, k-1$, we have that $B$ is $k$-semi-regular.

Applying this construction to $D^\dagger$ with $B^\dagger$ and \eqref{Eq.InfRegChar2} in place of \eqref{Eq.InfRegChar1} to the map $\tilde{g}$ with respect to the induced adapted operator $A^\ast$ for $(D')^\dagger$, which is none other than the adjoint map for $g$, we obtain that $B^\dagger$ is $k$-semi-regular also.
Therefore, $B$ is $k$-elliptically regular.
\end{proof}

\begin{corollary}
\label{Cor.InfReg}
Assume \ref{Hyp.StdFirst}--\ref{Hyp.StdLast}.
Let $B\subset \SobH{\frac12}(\dM;E)$ be a closed subspace.
Then the following are equivalent:
\begin{enumerate}
\item 
$B$ is a $\infty$-regular elliptic boundary condition for $D$;
\item
for all $j\in\N_0$ we have:
\begin{align*}
\dom&(D_{B,\max}) \cap \SobH[loc]{j+1}(M;E) \\
&= \set{u\in\dom(D_{B,\max}): Du\in\SobH[loc]{j}(M;F)\text{ and }u\rest{\dM}\in\SobH{j+\frac12}(M;E)} \,\,\text{and} \\
\dom&(D^\dagger_{B^\dagger,\max})\cap \SobH[loc]{j+1}(M;F)  \\ 
&= \set{u\in\dom(D^\dagger_{B^\dagger,\max}): D^\dagger u\in\SobH[loc]{j}(M;E)\text{ and }u\rest{\dM}\in\SobH{j+\frac12}(M;F)} .
\end{align*} 
\end{enumerate}
\end{corollary}

For a given adapted boundary operator $A$ for $D$, the condition $u\rest{\dM} \in \SobH{k+\frac{1}{2}}(\dM;E)$ in Theorem~\ref{Thm.HigherReg} and Corollary~\ref{Cor.InfReg} can be replaced by $\chi^+(A)u\rest{\dM} \in \SobH{k+\frac{1}{2}}(\dM;E)$, see Theorem~\ref{Thm:Trace}~\ref{Thm:Trace4}.


\alphsection{Ellipticity of the Rarita-Schwinger operator}
\label{appendix.RS}

We check that the Rarita-Schwinger operator $D_{\nicefrac32}$ is elliptic using the notation from Example~\ref{Ex:RS}.
The principal symbol is given by 
\begin{equation}
\sym_{D_{\nicefrac32}}(\xi) = (\id_{T^*M\otimes F} - \tilde\iota\circ\tilde\gamma)(\id_{T^*M}\otimes\sym_{D}(\xi))|_{E^{\nicefrac32}}.
\label{eq.RS.symbol}
\end{equation}
The Riemannian metric induces a map $T^*M\otimes T^*M\otimes E \to E$, $\xi\otimes\eta\otimes v\mapsto \xi\lrcorner (\eta\otimes v) = \inprod{\xi,\eta}v$.
Given a covector $\xi\in T^*_xM\setminus\set{0}$, we put
\begin{align*}
E^{\nicefrac32}_x(\xi) &:= \set{\Phi\in E^{\nicefrac32}_x : \xi\lrcorner\Phi=0}, \\
E^{\nicefrac32}_x(\xi)' &:= \set{(\id_{T^*M\otimes F} - \iota\circ\gamma)(\xi\otimes v) : v\in E_x}.
\end{align*}
We then have the orthogonal decomposition 
\begin{equation}
E^{\nicefrac32}_x = E^{\nicefrac32}_x(\xi) \oplus E^{\nicefrac32}_x(\xi)' 
\label{eq.RS.decomposition}
\end{equation}
because $E^{\nicefrac32}_x(\xi)$ is the kernel of the map $E^{\nicefrac32}_x\to E_x$, $\Phi\mapsto \xi\lrcorner\Phi=\xi\lrcorner(\id_{T^*M\otimes F} - \iota\circ\gamma)\Phi$, which is the adjoint of the map $E_x\to E_x^{\nicefrac32}$, $v\mapsto (\id_{T^*M\otimes F} - \iota\circ\gamma)(\xi\otimes v)$.

We compute $\sym_{D_{\nicefrac{3}{2}}}(\xi)^*\sym_{D_{\nicefrac{3}{2}}}(\xi)$ on the spaces $E^{\nicefrac32}_x$ and $E^{\nicefrac32}_x(\xi)'$ separately.
For $\Phi=\sum_i e^i\otimes v_i \in E^{\nicefrac32}_x(\xi)$ we have
\begin{align*}
\sym_{D_{\nicefrac{3}{2}}}&(\xi)^*\sym_{D_{\nicefrac{3}{2}}}(\xi)\Phi \\
&=
(\id_{T^*M\otimes E} - \iota\circ\gamma)(\id_{T^*M}\otimes\sym_{D}(\xi)^*)(\id_{T^*M\otimes F} - \tilde\iota\circ\tilde\gamma)^2(\id_{T^*M}\otimes\sym_{D}(\xi))\Phi \\
&=
(\id_{T^*M\otimes E} - \iota\circ\gamma)(\id_{T^*M}\otimes\sym_{D}(\xi)^*)(\id_{T^*M\otimes F} - \tilde\iota\circ\tilde\gamma)(\id_{T^*M}\otimes\sym_{D}(\xi))\Phi .
\end{align*}
Now
\begin{align*}
(&\id_{T^*M}\otimes\sym_{D}(\xi)^*)(\id_{T^*M\otimes F} - \tilde\iota\circ\tilde\gamma)(\id_{T^*M}\otimes\sym_{D}(\xi))\Phi \\
&=
(\id_{T^*M}\otimes\sym_{D}(\xi)^*)(\id_{T^*M}\otimes\sym_{D}(\xi))\Phi - (\id_{T^*M}\otimes\sym_{D}(\xi)^*)(\tilde\iota\tilde\gamma)\sum_i e^i\otimes \sym_D(\xi)v_i \\
&=
|\xi|^2\Phi - (\id_{T^*M}\otimes\sym_{D}(\xi)^*)\tilde\iota\sum_i \sigma_D(e^i)^* \sym_D(\xi)v_i\\
&=
|\xi|^2\Phi - (\id_{T^*M}\otimes\sym_{D}(\xi)^*)\tilde\iota\sum_i \big(-\sigma_D(\xi)^* \sym_D(e^i)+2\inprod{e^i,\xi}\big)v_i\\
&=
|\xi|^2\Phi - (\id_{T^*M}\otimes\sym_{D}(\xi)^*)\tilde\iota(-\sigma_D(\xi)^*\gamma\Phi + 2\xi\lrcorner\Phi)\\
&=
|\xi|^2\Phi - 2(\id_{T^*M}\otimes\sym_{D}(\xi)^*)\tilde\iota(\xi\lrcorner\Phi) .
\end{align*}
Hence
\begin{align*}
\sym_{D_{\nicefrac{3}{2}}}(\xi)^*\sym_{D_{\nicefrac{3}{2}}}(\xi)\Phi
&=
(\id_{T^*M\otimes E} - \iota\circ\gamma)(|\xi|^2\Phi - 2(\id_{T^*M}\otimes\sym_{D}(\xi)^*)\tilde\iota(\xi\lrcorner\Phi))\\
&=
|\xi|^2\Phi - 2(\id_{T^*M\otimes E} - \iota\circ\gamma)(\id_{T^*M}\otimes\sym_{D}(\xi)^*)\tilde\iota(\xi\lrcorner\Phi).
\end{align*}
Now if $\Phi \in E^{\nicefrac32}_x(\xi)$ we get $\sym_{D_{\nicefrac{3}{2}}}(\xi)^*\sym_{D_{\nicefrac{3}{2}}}(\xi)\Phi=
|\xi|^2\Phi$. 
For $\Phi= (\id_{T^*M\otimes F} - \iota\circ\gamma)(\xi\otimes v)\in E^{\nicefrac32}_x(\xi)'$ we compute 
\begin{align*}
\tilde\iota(\xi\lrcorner\Phi) 
&=
\tilde\iota(\xi\lrcorner(\xi\otimes v - \iota\sym_D(\xi)v)) \\
&=
\tilde\iota\Big(|\xi|^2 v - \tfrac{1}{n}\sum_i\inprod{\xi,e^i}\sym_D(e^i)^*\sym_D(\xi)v\Big) \\
&=
\tilde\iota\Big(|\xi|^2 v - \tfrac{1}{n}\sym_D(\xi)^*\sym_D(\xi)v\Big) \\
&=
\tfrac{n-1}{n}\,|\xi|^2\, \tilde\iota(v) \\
&=
\tfrac{n-1}{n^2}\,|\xi|^2\, \sum_i e^i\otimes \sym_D(e^i) v .
\end{align*}
Therefore,
\begin{align*}
(&\id_{T^*M\otimes E} - \iota\circ\gamma)(\id_{T^*M}\otimes\sym_{D}(\xi)^*)\tilde\iota(\xi\lrcorner\Phi) \\
&=
\tfrac{n-1}{n^2}\,|\xi|^2\,(\id_{T^*M\otimes E} - \iota\circ\gamma)(\id_{T^*M}\otimes\sym_{D}(\xi)^*) \sum_i e^i\otimes \sym_D(e^i) v\\
&=
\tfrac{n-1}{n^2}\,|\xi|^2\,(\id_{T^*M\otimes E} - \iota\circ\gamma)\sum_i e^i\otimes \sym_{D}(\xi)^*\sym_D(e^i) v\\
&=
\tfrac{n-1}{n^2}\,|\xi|^2\, \sum_i\Big( e^i\otimes \sym_{D}(\xi)^*\sym_D(e^i)  - \iota \sym_D(e^i) \sym_{D}(\xi)^*\sym_D(e^i) \Big)v \\
&=
\tfrac{n-1}{n^2}\,|\xi|^2\, \sum_i\Big( e^i\otimes (-\sym_{D}(e^i)^*\sym_D(\xi) + 2\inprod{\xi,e^i})  - \iota \big(-\sym_D(\xi) \sym_{D}(e^i)^* \\
&\qquad\qquad +2\inprod{\xi,e^i}\big)\sym_D(e^i)  \Big)v \\
&=
\tfrac{n-1}{n^2}\,|\xi|^2\, \big( -n\iota\sym_D(\xi) + 2\xi\otimes\cdot  + n \iota \sym_D(\xi)-2\iota\sym_D(\xi)  \big)v \\
&=
2\tfrac{n-1}{n^2}\,|\xi|^2\, \big(  \xi\otimes v  -\iota\sym_D(\xi)v  \big) \\
&=
2\tfrac{n-1}{n^2}\,|\xi|^2\, \Phi .
\end{align*}
Hence
\begin{align*}
\sym_{D_{\nicefrac{3}{2}}}(\xi)^*\sym_{D_{\nicefrac{3}{2}}}(\xi)\Phi
&=
\big(1-2\cdot 2\tfrac{n-1}{n^2}\big)|\xi|^2\, \Phi
=
\Big(\frac{n-2}{n}\Big)^2 |\xi|^2\, \Phi .
\end{align*}
Thus $\sym_{D_{\nicefrac{3}{2}}}(\xi)^*\sym_{D_{\nicefrac{3}{2}}}(\xi)$ has the eigenvalues $|\xi|^2$ and $\big(\frac{n-2}{n}\big)^2 |\xi|^2$ and is therefore invertible if $\xi\neq0$.
This shows that $D_{\nicefrac{3}{2}}$ is not a Dirac-type operator but it is elliptic.


\setlength{\parskip}{0pt}
\end{document}